\documentclass[12pt]{amsart}

\usepackage[usenames,dvipsnames,svgnames,table]{xcolor}

\usepackage{xargs}
\usepackage[colorinlistoftodos,prependcaption,textsize=tiny,linecolor=red,backgroundcolor=red!25,bordercolor=red]{todonotes}
\newcommandx{\nantel}[2][1=]{\todo[linecolor=brown,backgroundcolor=red!25,bordercolor=red,#1]{#2 ---Nantel}}
\newcommandx{\caro}[2][1=]{\todo[linecolor=purple,backgroundcolor=brown!25,bordercolor=brown,#1]{#2 ---C}}
\newcommandx{\todocite}[2][1=]{\todo[linecolor=green,backgroundcolor=green!25,bordercolor=green,#1]{cite: #2}}
\newcommandx{\franco}[2][1=]{\todo[linecolor=blue,backgroundcolor=blue!25,bordercolor=blue,#1]{#2 ---Franco}}

\newtheorem{theorem}{Theorem}
\newtheorem{proposition}[theorem]{Proposition}

\newtheorem{lemma}[theorem]{Lemma}
\theoremstyle{definition}
\newtheorem{remark}[theorem]{Remark}
\newtheorem{example}[theorem]{Example}
\newtheorem{definition}[theorem]{Definition}

\newtheorem{problem}[theorem]{Problem}

\usepackage{amssymb}
\usepackage{amsmath}
\usepackage{amsthm}
\usepackage{amsfonts}
\usepackage{amscd}
\usepackage{mathtools}
\usepackage{mathrsfs}
\usepackage[mathscr]{eucal}

\usepackage{listings, lstautogobble}
\lstnewenvironment{sagecode}{}{}

\usepackage{vmargin}
\setpapersize{USletter}
\setmargrb{2.5cm}{2cm}{2.5cm}{2.cm}
\usepackage{version}
\excludeversion{sagecode}

\usepackage[inline]{enumitem}

\usepackage{tikz}
\usepackage{tkz-graph}
\usepackage{tkz-berge}
\usetikzlibrary{arrows,shapes}

\newcommand{\NN}{{\mathbb N}}

\newcommand{\RR}{{\mathbb R}}
\newcommand{\R}{\mathscr{R}}
\newcommand{\IP}{\operatorname{IP}}
\newcommand{\BP}{\operatorname{BP}}
\newcommand{\SSP}{\ensuremath{\operatorname{SSP}}}
\newcommand{\rank}{\operatorname{rank}}
\newcommand{\Bases}{\mathcal B}
\newcommand{\Poly}{{\mathcal P}}
\newcommand{\Cliq}{\operatorname{Cliq}}

\newcommand{\NNP}{\operatorname{NN}}
\newcommand{\NCP}{\operatorname{NC}}
\newcommand{\Stab}{\operatorname{Stab}}

\def\b{\textcolor{blue}}
\def\v{\textcolor{green!60!black}}

\definecolor{cof}{RGB}{219,144,71}
\definecolor{pur}{RGB}{186,146,162}
\definecolor{greeo}{RGB}{91,173,69}
\definecolor{greet}{RGB}{52,111,72}

\newcommand{\defcolor}[1]{{\color{DodgerBlue}#1}}
\newcommand{\demph}[1]{\defcolor{{\sl #1}}}

\def\Ind{\mathcal I}
\newcommandx\conv{\operatorname{conv}}

\title[Stable set polytopes]{Stable set polytopes and their 1-skeleta}

\author[]{%
    Farid Aliniaeifard%
    \email{farid@math.ubc.ca}%
    \address{Department of Mathematics, The University of British Columbia, Vancouver BC V6T 1Z2, Canada}
    \and
    Carolina Benedetti%
    \email{{c.benedetti@uniandes.edu.co}}%
    \address{Departamento de Matem\'aticas, Universidad de los Andes, Bogot\'a}
    \and
    Nantel Bergeron%
    \email{bergeron@mathstat.yorku.ca}
    \address{Department of Mathematics and Statistics, York University, Toronto ON M3J 1P3, Canada}
    \and
    Shu Xiao Li%
    \email{lishuxiao@dlut.edu.cn}%
    \address{School of Mathematical Science, Dalian University of Technology, Dalian, China}
    \and
    Franco Saliola%
    \email{saliola.franco@uqam.ca}
    \address{LACIM, Depart{\'e}ment de math{\'e}matiques, Universit{\'e} du Qu{\'e}bec {\`a} Montr{\'e}al, Montr{\'e}al QC, Canada}
}
    \thanks{Benedetti: Research supported by FAPA grant, Faculty of Science Universidad de los Andes.}
    \thanks{Bergeron: Research supported in part by NSERC and York Research Chair.}%
    \thanks{Saliola: Research supported in part by NSERC.}%

\begin{document}

\begin{abstract}
    We characterize the edges of two classes of $0/1$-polytopes. The first
    class corresponds to the stable set polytope of a graph $G$ and  includes
    chain polytopes of posets, some instances of matroid independence
    polytopes, as well as newly-defined polytopes whose vertices correspond to
    noncrossing set
    partitions. In analogy with matroid basis polytopes, the second class is
    obtained by considering the stable sets of maximal cardinality.
    We investigate how the class of $0/1$-polytopes whose edges satisfy our
    characterization is situated within the hierarchy of $0/1$-polytopes.
    This includes the class of matroid polytopes. We also study the diameter of
    these classes of polytopes and improve slightly on the Hirsch bound.
\end{abstract}


\maketitle

\setcounter{tocdepth}{2}
\tableofcontents

\section{Introduction}

The family of $0/1$-polytopes consists of polytopes in $\RR^n$ whose vertices contain only entries 0 or 1.
We study several classes of $0/1$-polytopes, with a focus on those associated with a graph $G$. The first one
is the \demph{stable set polytope} of $G$ (also known as the \demph{vertex
packing polytope} in the literature) whose vertices are indexed by the \emph{stable
sets} of $G$ (see  \cite[Chap 9]{GLS93}).
This class includes several polytopes arising in algebraic combinatorics such
as the \emph{chain polytope of a poset}, some instances of \emph{matroid
independence polytopes}, and the \emph{unipotent polytopes} introduced in
\cite{Thiem-FPSAC2017, Thiem2018}. We also identify a family of
polytopes whose vertices correspond to noncrossing set partitions.
In analogy with the relationship between matroid independence polytopes and
\emph{matroid basis polytopes}, we study the polytope whose vertices correspond
to stable sets of maximal cardinality. This includes as
a special case the \emph{Birkhoff polytopes}.

Our contributions are described as follows.
\begin{enumerate}[wide, topsep=1ex, itemsep=1ex]
    \item
        We present a new characterization of the edges and $1$-skeleta of stable set
        polytopes (see Theorem~\ref{RP1skeleton}).
        Chv\'atal gives in \cite{Chvatal1975}  a beautiful characterization of these edges as well, although that  description
        is valid only for stable sets of graphs and does not generalize easily to other $0/1$-polytopes.
        We will show in Section~\ref{hierarchy} that our
        characterization extends to other classes of $0/1$-polytopes.
        Our result can be proved using Chv\'atal's result,
        but we give a direct proof that uses a novel
        characterization of the edges of a polytope (see Lemma~\ref{lem:edge}).

    \item
        Among the family of stable set polytopes, we identify two new families,
        the nonnesting polytopes $\NNP_n$ (Section~\ref{sec:NNP}) and
        the noncrossing polytopes $\NCP_n$ (Section~\ref{sec:NCP}), whose
        vertices are indexed by nonnesting
        and noncrossing set partitions of $[n]$, respectively.
        In addition to describing their $1$-skeleta via
        Theorem~\ref{RP1skeleton}, we describe some of their facets (see
        Section~\ref{facets-NNP} and Section~\ref{facets-NCP}, respectively).

    \item
        We investigate how the class of $0/1$-polytopes whose edges
        satisfy our characterization is situated within the hierarchy of
        $0/1$-polytopes.
        We show that this class is properly contained in the class of all
        $0/1$-polytopes and that it properly contains the stable set
        polytopes, the matroid basis polytopes, and the matroid independent set
        polytopes
        (see Figure~\ref{inclusions-polytope-classes} and the results of
        Section~\ref{hierarchy}).
        We also characterize the intersection of the class of stable set
        polytopes and the class of independent set polytopes of matroids.
  Finally, the family of simplicial complex polytopes (see Section~\ref{sec:SCP}) provides examples of $0/1$-polytopes not always satisfying criterion (E) as given in Section \ref{hierarchy}.
        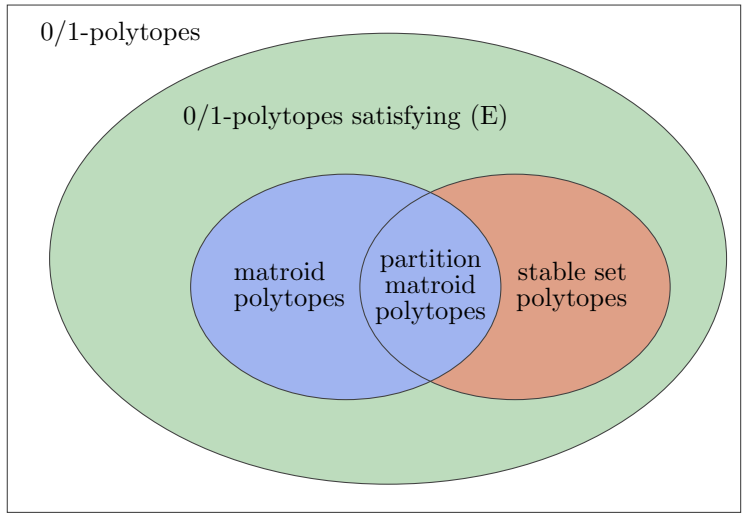
\begin{figure}[htpb]
            \begin{tikzpicture}[scale=0.75]
                \draw[Black!75, fill=White] (-4, -4) rectangle (9, 5);
                \begin{scope}[blend group = hard light]
                   \fill[Black!75, fill=ForestGreen!30!white] (2.75, 1) ellipse (5.7 and 3.7);
                   \draw[Black!75, dashed, thick, fill=red!20!white] (3.2, -0.2) ellipse (5.5 and 3.5);
                \end{scope}
                \begin{scope}[blend group = hard light]
                    \fill[BrickRed!40]   (5, 0) ellipse (2.75 and 2);
                    \fill[RoyalBlue!50]  (2, 0) ellipse (2.75 and 2);
                \end{scope}
                \draw[Black!75]   (5, 0) ellipse (2.75 and 2);
                \draw[Black!75]  (2, 0) ellipse (2.75 and 2);
                \node at (-2, 4.5)  {\footnotesize $0/1$-polytopes};
                \node at (2.5, 3.7)     {\footnotesize $0/1$-polytopes satisfying \eqref{1-skeleton-condition}};
                \node at (3.1, 2.3)     {\footnotesize simplicial complex polytopes};
                \node at (1, 0)     {\footnotesize \begin{tabular}{l}matroid \\[-0.5ex] polytopes\end{tabular}};
                \node at (3.5, 0)   {\footnotesize \begin{tabular}{c}partition \\[-0.5ex] matroid \\[-0.5ex] polytopes\end{tabular}};
                \node at (6, 0)     {\footnotesize \begin{tabular}{r}stable set \\[-0.5ex] polytopes\end{tabular}};
            \end{tikzpicture}
            \caption{The classes of $0/1$-polytopes studied in
            this paper. Property (E) is defined in Section~\ref{hierarchy}.}
            \label{inclusions-polytope-classes}
        \end{figure}

    \item
        In Section~\ref{section-on-the-diameter}, we study the Hirsch
        conjecture as it pertains to our setting. Recall that the \emph{Hirsch
        conjecture} asserts that the diameter of every $d$-dimensional (convex)
        polytope with $n$ facets is at most $n - d$ (see
        Section~\ref{section-on-the-diameter} for definitions). It is related
        to the \emph{travelling salesman problem} and the \emph{simplex method}
        as it provides an easy-to-compute bound on the minimum distance between
        any two vertices. Although the Hirsch conjecture is false in
        general~\cite{Santos2012}, it is true for
        $0/1$-polytopes~\cite{Naddef1989}, and we provide an improvement on this
        bound for some of the polytopes we study here.

    \item
        In Section~\ref{sec:OpenProb}, we conclude our investigation with
        a discussion of some open problems and conjectures.
\end{enumerate}

\section{$0/1$-polytopes, simplicial complex polytopes,
      and stable set polytopes}
      \label{sec:polytopes}

\subsection{Indicator vectors}
Let $X$ be a finite set and let $\RR^{X}$ denote a real vector space with
standard basis, denoted $\{e_x : x \in X\}$, whose elements are indexed by the
elements of $X$.
We associate an element $e_A$ of $\RR^{X}$ to each subset $A \subseteq X$ as
follows: define the \demph{indicator vector} of $A$ as
\begin{equation*}
    e_A = \sum_{a \in A} e_a \in \RR^{X}.
\end{equation*}
Note that $e_{\emptyset} = 0 \in \RR^{X}$.

It is often convenient to identify $\RR^{X}$ with $\RR^{|X|}$.
To do so, fix any total order $(x_1, x_2, \ldots, x_n)$ on $X$ and identify the
basis vector $e_{x_i} \in \RR^X$ with the standard basis vector $e_i \in
\RR^{|X|}$.

We will also make use of the usual inner product $\langle \cdot, \cdot
\rangle$ on $\RR^{X}$ for which $\{e_x : x \in X\}$ is an orthonormal basis.
Thus,
\begin{equation*}
    \langle e_x, e_A \rangle
    =
    \begin{cases}
        1, & \text{if $x \in A$,} \\
        0, & \text{if $x \notin A$.}
    \end{cases}
\end{equation*}
The following straightforward consequence will be used several times:
\begin{equation*}
    \langle e_A - e_B, e_x \rangle =
    \begin{cases}
        0, & \text{if $x \in A \cap B$ or $x \notin A \cup B$,} \\
        1, & \text{if $x \in A \setminus B$,} \\
        -1, & \text{if $x \in B \setminus A$.}
    \end{cases}
\end{equation*}

\subsection{$0/1$-polytopes and simplicial complex polytopes} \label{sec:SCP}
A  \demph{$0/1$-polytope} in $\RR^X$ is the convex hull of the indicator
vectors of the sets in a set $\mathcal C$ of subsets of $X$:
        $$\Poly_{\mathcal C}=\conv\left\{ e_A : A \in {\mathcal C} \right\} \subseteq \RR^X\,.$$
The set $\mathcal C$ is a (abstract) \demph{simplicial complex} if for any $ B\in {\mathcal C}$ and $A\subseteq B$ it follows that  $A\in{\mathcal C}$. In this case, we say that $\Poly_{\mathcal C}$ is a \demph{simplicial complex polytope}.
 In the following, we are interested in particular families of simplicial complex polytopes and their maximal faces.

\subsection{Stable set polytopes (SSP)}

Let $G = (V,E)$ be a \demph{simple} graph, that is, $G$ has no loops and no multiple edges.
A subset $A$ of the vertices $V$ is \demph{stable} for $G$ if no two
vertices in $A$ are connected by an edge in $G$.
Let \demph{$\Stab(G)$} denote the set of stable sets of $G$.
  It then follows that $\Stab(G)$ is a simplicial complex.
The \demph{stable set polytope} of $G$ is the convex hull of the indicator
vectors of the stable sets of $G$, that is
\begin{equation*}
    \SSP(G) = \Poly_{\Stab(G)} \subseteq \RR^V.
\end{equation*}

\subsection{Examples}
\label{sec:examples}

Our motivation for studying this family of polytopes is the vast variety of
polytopes that can be realized as stable set polytopes.

\subsubsection{Polytope of independent sets of a relation}
\label{sec:IP(R)}

Let $\R \subseteq X^2$ be a relation on a finite set $X$. A subset $A$ of $X$
is \demph{independent for $\R$} if and only if $(x, y) \notin \R$ and $(y, x)
\notin \R$ for all \emph{distinct} $x, y \in A$.
Let
\begin{equation*}
    \defcolor{\Ind(X, \R)} = \left\{ A \subseteq X : \text{$A$ is independent for $\R$} \right\}.
\end{equation*}
Note that since we require $x$ and $y$ to be \emph{distinct}, it follows that
$\{x\}$ is independent for all $x \in X$.
Note also that if $A$ is independent for $\R$, then every subset of $A$ is also
independent for $\R$.
Define the \demph{independent set polytope $\IP(\R)$} of a relation $\R$ to
be the convex hull of the indicator vectors of the independent sets for $\R$.

Note that $\IP(\R)$ is a special case of a stable set polytope.
Let $G_\R$ be the simple graph with vertex set $X$ and with edge set consisting of
$\{x, y\}$ if and only if $(x, y) \in \R$ or $(y, x) \in \R$. (Implicit in this definition
is the fact that $x$ and $y$ are distinct.) Note that a subset $A \subseteq X$
is stable for $G_\R$ if and only if $A$ is independent for $\R$. Consequently,
\begin{equation*}
    \IP(\R)=\SSP(G_\R).
\end{equation*}

\begin{example}
    Take $X = \{1,2,3\}$ and $\R = \{(1, 2), (2, 3)\}$.
    The independent sets for $\R$ are $\left\{\emptyset, \{1\}, \{2\}, \{3\}, \{1,3\}\right\}$
    so that
    \begin{equation*}
        \IP(\R) = \conv\{(0,0,0), (1,0,0), (0,1,0), (0,0,1), (1,0,1)\}.
    \end{equation*}
    \begin{sagecode}
        sage: X = {1, 2, 3}
        sage: R = {(1, 2), (2, 3)}
        sage: T = {A for A in Subsets(X) if all((x, y) not in R for x in A for y in A)}
        sage: T
        {{1}, {3}, {}, {2}, {1, 3}}
    \end{sagecode}
\end{example}

\subsubsection{$n$-cube}
\label{sec:n-cube}
Let $([n], \emptyset)$ be a graph with $n$ vertices and no edges. Then every
subset of of the vertices is stable, and the associated polytope is the
\demph{$n$-cube};
\emph{i.e.,} $\SSP([n], \emptyset) = \conv(\{0,1\}^n)$.

\subsubsection{Chain polytope of a poset}
\label{sec:chain-polytope}
Let $P = (X, \preceq)$ be a finite poset.
The \demph{comparability graph $G_P$} of $P$ is the graph whose vertex
set is $X$ and which contains an edge connecting $x$ and $y$ if and only if $x \prec y$ or $y \prec x$.
A subset $A \subseteq X$ is stable for $G_P$ if and only if it is an antichain of the poset.
Hence, $\SSP(G_P)$ is the \demph{chain polytope} of $P$
originally introduced by R.~Stanley in \cite{Stanley1986}.


\subsubsection{Bell polytopes}
\label{sec:bell-polytope}
Let $G$ be the graph with vertex set $X_n=\{(i,j): 1\le i<j\le n\}$ and with an
edge connecting $(i, j)$ and $(k, l)$ if and only if
\begin{equation*}
    \text{$i = k$ and $j \neq l$}
    \qquad\text{or}\qquad
    \text{$i \neq k$ and $j = l$.}
\end{equation*}

The stable sets of $G$, and hence the vertices of $\SSP(G)$, can be identified
with set partitions of the set $[n]:=\{1,\dots,n\}$, as follows.
Identify $e_{(i,j)} \in \RR^{X_n}$ with the upper triangular $n \times n$
matrix whose $(i, j)$ entry is $1$ and whose other entries are $0$.
Then $e_A$ is identified with a strictly upper triangular $0/1$-matrix.
If $A$ is stable for $G$, then the matrix $e_A$ has at most one $1$ in
each row and column.
We can encode such a matrix by a set partition $S = \{S_1,\dots,S_{\ell}\}$
of $[n]$ by placing $i$ and $j$ in the same set $S_r$ if the $(i, j)$ entry of
the matrix is $1$.

This polytope, which we call the \demph{Bell polytope $B_n$}, is a particular case
of the \demph{unipotent polytopes} introduced in \cite{Thiem-FPSAC2017,
Thiem2018}.

\subsubsection{Nonnesting (partition) polytope}
\label{sec:NNP}
The \demph{nonnesting polytope $\NNP_n$} is the stable set polytope of the
comparability graph of the root poset of type $A_n$, which we think of
as $X_n = \{(i,j): 1\le i<j\le n\}$ with the following relation
(NB. this relation is different from the one above):
\begin{equation*}
    (i, j) \leq (k, l)
    \qquad\text{if and only if}\qquad
    k \leq i < j \leq l.
\end{equation*}
As above, the stable sets for the comparability graph of this poset are also
encoded by certain strictly upper triangular matrices with at most one $1$ in
each row and column;
or equivalently, by certain set partitions of $[n]$.
It turns out that we obtain precisely the nonnesting partitions of $[n]$
in this way.
See Example~\ref{ex_uni3}.
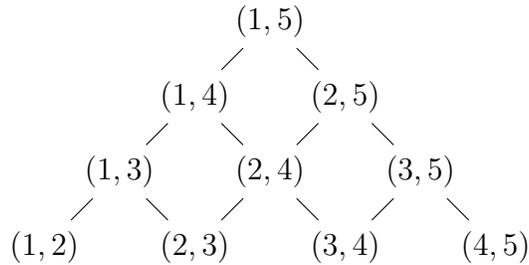
\begin{figure}[htpb]
    \centering
    \begin{tikzpicture}[>=latex,line join=bevel]
        \node (node_0) at (0, 0) [draw,draw=none] {$\left(1, 2\right)$};
        \node (node_1) at (1, 1) [draw,draw=none] {$\left(1, 3\right)$};
        \node (node_2) at (2, 2) [draw,draw=none] {$\left(1, 4\right)$};
        \node (node_3) at (3, 3) [draw,draw=none] {$\left(1, 5\right)$};
        \node (node_4) at (2, 0) [draw,draw=none] {$\left(2, 3\right)$};
        \node (node_5) at (3, 1) [draw,draw=none] {$\left(2, 4\right)$};
        \node (node_6) at (4, 2) [draw,draw=none] {$\left(2, 5\right)$};
        \node (node_7) at (4, 0) [draw,draw=none] {$\left(3, 4\right)$};
        \node (node_8) at (5, 1) [draw,draw=none] {$\left(3, 5\right)$};
        \node (node_9) at (6, 0) [draw,draw=none] {$\left(4, 5\right)$};
        \draw [black] (node_0) to (node_1);
        \draw [black] (node_1) to (node_2);
        \draw [black] (node_2) to (node_3);
        \draw [black] (node_4) to (node_5);
        \draw [black] (node_5) to (node_6);
        \draw [black] (node_7) to (node_8);
        \draw [black] (node_4) to (node_1);
        \draw [black] (node_5) to (node_2);
        \draw [black] (node_6) to (node_3);
        \draw [black] (node_7) to (node_5);
        \draw [black] (node_8) to (node_6);
        \draw [black] (node_9) to (node_8);
    \end{tikzpicture}
    \caption{Root poset of type $A_5$}
\end{figure}

\subsubsection{Noncrossing (partition) polytope}
\label{sec:NCP}
The \demph{noncrossing polytope $\NCP_n$} is the stable set polytope of
the graph on $X_n = \{(i,j): 1\le i<j\le n\}$ with edges
connecting $(i,j)$ and $(k,l)$ if and only if
\begin{equation*}
    \text{$i = k$ and $j \neq l$}
    \quad\text{or}\quad
    \text{$i \neq k$ and $j = l$}
    \quad\text{or}\quad
    \text{$i < k < j < l$}.
\end{equation*}
The stable sets for this graph are also encoded by certain strictly
upper triangular matrices with at most one $1$ in each row and column; or
equivalently, by certain set partitions of $[n]$. It turns out that we obtain
precisely the noncrossing partitions of $[n]$ in this way.

\begin{example}\label{ex_uni3}
    For $n \leq 3$, the Bell polytope $B_n$, the nonnesting polytope $\NNP_n$ and
    the noncrossing polytopes $\NCP_n$ coincide as every set partition of $[3]$
    is noncrossing and nonnesting.
    For example, when $n = 3$ we have the graph $G = (V, E)$, where
    \begin{equation*}
        \begin{aligned}
            V &= \Big\{ (1,2), (1,3), (2,3) \Big\}
            \\
            E &= \Big\{ \{(1,2), (1,3)\}, \{(1,3), (2,3)\} \Big\}
            \\
            \SSP(G) &= \conv\Big\{(0,0,0), (1,0,0), (0,1,0), (0,0,1), (1,0,1)\Big\}.
        \end{aligned}
    \end{equation*}
    \begin{figure}
        \centering
        \begin{tikzpicture}[thick,scale=2]
            \coordinate (A0) at (0,0);
            \coordinate (A1) at (-.8,-.5);
            \coordinate (A2) at (1,0);
            \coordinate (A3) at (0,1);
            \coordinate (A4) at (-.8,.5);

            \begin{scope}[thick,dashed,,opacity=0.6]
            \draw (A1) -- (A0) -- (A2) ;
            \draw (A0) -- (A3);
            \end{scope}
            \node at (A0) {\tiny{
            \begin{tikzpicture}
            \node at (0.15,0.15) {\tiny{0}};
            \node at (.45,0.15) {\tiny{0}};
            \node at (.45,-.15) {\tiny{0}};
            \draw[step=1cm,scale=.3,black,] (0,0) grid (2,1);
            \draw[step=1cm,scale=.3,black,] (1,0) grid (2,-1);
            \end{tikzpicture}
            }};
            \node at (A1) {\tiny{\begin{tikzpicture}
            \node at (0.15,0.15) {\tiny{1}};
            \node at (.45,0.15) {\tiny{0}};
            \node at (.45,-.15) {\tiny{0}};
            \draw[step=1cm,scale=.3,black,] (0,0) grid (2,1);
            \draw[step=1cm,scale=.3,black,] (1,0) grid (2,-1);
            \end{tikzpicture}}};
            \node at (A2) {\tiny{\begin{tikzpicture}
            \node at (0.15,0.15) {\tiny{0}};
            \node at (.45,0.15) {\tiny{1}};
            \node at (.45,-.15) {\tiny{0}};
            \draw[step=1cm,scale=.3,black,] (0,0) grid (2,1);
            \draw[step=1cm,scale=.3,black,] (1,0) grid (2,-1);
            \end{tikzpicture}}};
            \node at (A3) {\tiny{\begin{tikzpicture}
            \node at (0.15,0.15) {\tiny{0}};
            \node at (.45,0.15) {\tiny{0}};
            \node at (.45,-.15) {\tiny{1}};
            \draw[step=1cm,scale=.3,black,] (0,0) grid (2,1);
            \draw[step=1cm,scale=.3,black,] (1,0) grid (2,-1);
            \end{tikzpicture}}};
            \node at (A4) {\tiny{\begin{tikzpicture}
            \node at (0.15,0.15) {\tiny{1}};
            \node at (.45,0.15) {\tiny{0}};
            \node at (.45,-.15) {\tiny{1}};
            \draw[step=1cm,scale=.3,black,] (0,0) grid (2,1);
            \draw[step=1cm,scale=.3,black,] (1,0) grid (2,-1);
            \end{tikzpicture}}};
            \draw[fill=greeo,opacity=0.4] (A2) -- (A4) -- (A3);
            \draw[fill=pur,opacity=0.4] (A1) -- (A4) -- (A2);
            \draw (A1) -- (A2) -- (A3) -- (A4) --cycle;
            \draw (A2) -- (A4);
        \end{tikzpicture}
        \caption{The polytopes $B_n$, $\NNP_n$ and $\NCP_n$ coincide for $n=3$. The
            vertices are labelled by the upper triangular portion of the $3
            \times 3$ strictly upper triangular $0/1$ matrices with at most one
            $1$ in each row and each column.}
    \end{figure}
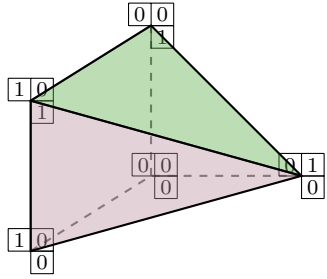
\end{example}

\begin{remark}
    The polytopes $B_n$, $\NNP_n$ and $\NCP_n$ admit
    generalizations to other types of root systems.
    The type $B$ analogues were studied by Aaron Allen \cite{Allen2017}.
    See Section~\ref{facets-Bell-polytopes-B} for more information.
    We will see below that it is related to the Birkhoff polytopes.
\end{remark}

\subsection{Polytopes associated to a matroid}
\label{sec:matroid-independence-polytope}
A \demph{matroid} $M$ on a finite set $X$ is a non-empty collection $\Ind$
of subsets of $X$ satisfying:
\begin{enumerate}[noitemsep, label=(I\arabic*), topsep=2pt]
     \item\label{matroid-axiom-0}
        $\emptyset \in\mathcal I$;
    \item\label{matroid-axiom-1}
        if $A \in \Ind$ and $B \subseteq A$, then $B \in \Ind$; and
    \item\label{matroid-axiom-2}
        if $A, B \in \Ind$ and $|B| > |A|$, then there exists
        $b \in B \setminus A$ such that $A \cup \{b\} \in \Ind$.
\end{enumerate}

 We see from \ref{matroid-axiom-1} that $\Ind$ is a simplicial complex.

\subsubsection{Matroid independence polytope}
The elements of $\Ind$ are called the \demph{independent sets} of $M$.
The \demph{matroid independence polytope} of $M$ is $\Poly_{\Ind}$.
This family of polytopes was introduced by Edmonds in \cite{Edmonds1970} where
he also described the facet inequalities.

The independent sets of a relation $\R$ on $X$ satisfy \ref{matroid-axiom-1}, but not
necessarily \ref{matroid-axiom-2}. When a relation $\R$ satisfies both
\ref{matroid-axiom-1} and \ref{matroid-axiom-2}, the polytope $\IP(\R)$
defined in Section~\ref{sec:IP(R)}
coincides with the matroid independence polytope of a matroid. In this case,
results about matroid polytopes can be used to describe various aspects of
$\IP(\R)$.

\subsubsection{Matroid basis polytope}
\label{sec:matroid-basis-polytope}
The \demph{bases} of a matroid $M$ are the independent sets of $M$ that are
maximal with respect to inclusion. Let \demph{$\BP(M)$} be the polytope whose
vertices are the indicator vectors for the bases of $M$. By
\ref{matroid-axiom-2}, all
bases of $M$ have the same cardinality, which is called the \demph{rank} of $M$.
Note that $\BP(M)$ is the facet of the independent set polytope of $M$
supported by the hyperplane of vectors whose coordinates sum to the rank of
$M$.

In Section~\ref{sec:birkhoff}, we consider a generalization of this
construction: the face of $\SSP(G)$ supported by the hyperplane of vectors
whose coordinates sum to the maximal cardinality of a stable set of $G$.
This includes the Birkhoff polytopes as a special case.

\section{The $1$-skeleton of stable set polytopes}

Recall that the \demph{$1$-skeleton} of a polytope $P$ is the graph whose
vertices correspond to the $0$-dimensional faces of $P$; and there is an
edge connecting two vertices of the graph if and only if they are the vertices of
a $1$-dimensional face of $P$.
One of our main results is the following description of the $1$-skeleton of
the stable set polytope of a graph $G$.

Recall that $\Stab(G)$ denotes the stable sets of $G$.
\begin{theorem}
    \label{RP1skeleton}
    Let $\SSP(G)$ be the stable set polytope of a finite simple graph $G = (V, E)$.
    \begin{enumerate}[noitemsep]
        \item\label{RP1skeleton-vertices}
            The vertex set of $\SSP(G)$ is $\{e_A : A \in \Stab(G) \}$.

        \item\label{RP1skeleton-edges}
            Two distinct vertices $e_A$ and $e_B$ form an edge in $\SSP(G)$ if and only if
            for all $C, D \in \Stab(G)$, we have
            \begin{equation*}
                \text{$e_A + e_B = e_C + e_D$ implies $\{A, B\} = \{C, D\}$.}
            \end{equation*}
    \end{enumerate}
\end{theorem}

\begin{proof}[Proof of Theorem~\ref{RP1skeleton}, Part~\eqref{RP1skeleton-vertices}]
    Note that $e_A$ is not a nontrivial convex combination of the other $e_B$,
    for otherwise we would have a nontrivial convex combination of the
    vertices of the $|X|$-cube (since each $e_A$ is a vertex of the $|X|$-cube).
    Hence, $e_A$ is a vertex of $\SSP(G)$.
\end{proof}

The proof of part~\eqref{RP1skeleton-edges} of Theorem~\ref{RP1skeleton}
will make use of the following characterization of the edges of a polytope.
To our knowledge this characterization has not appeared in the literature.

\begin{lemma}\label{lem:edge}
    Two distinct vertices $a$ and $b$ of a polytope $P$ are \emph{not} the
    vertices of an edge of $P$ if and only if there exist $k \geq 1$ vertices $v_1, \dots, v_k$
    of $P$, distinct from $a, b$, and coefficients $\gamma_1, \dots,
    \gamma_k > 0$ such that
    \begin{equation*}
        a-b = \sum_{i=1}^k \gamma_i(v_i-b).
    \end{equation*}
\end{lemma}

\begin{proof}
    Suppose $a - b = \sum_{i=1}^k \gamma_i(v_i - b)$ with $\gamma_1, \dots,
    \gamma_k > 0$, where $v_1, \dots, v_k$ are $k \geq 1$ vertices of $P$ that
    are distinct from $a$ and $b$.
    Let $F$ denote the smallest face of $P$ containing $a$ and $b$,
    and let
    $H = \{u \in \RR^d: \langle u, c \rangle = c_0 \}$
    be a supporting hyperplane of $F$ satisfying
    $P \subseteq \{ u \in \RR^d : \langle u, c \rangle \geq c_0 \}$.
    Hence, for any vertex $v$ of $P$, we have $\langle v, c \rangle \geq c_0$,
    with equality if and only if $v \in F$.
    Thus,
    \begin{equation*}
        \langle a, c \rangle = c_0 = \langle b, c \rangle
        \qquad\text{and}\qquad
        \langle v - b, c \rangle \geq c_0 - c_0 = 0,
    \end{equation*}
    Since $a - b = \sum_{i=1}^k \gamma_i(v_i - b)$, we have
    \begin{equation*}
        0 = \langle a - b, c \rangle
        = \sum_{i=1}^k \gamma_i \langle v_i - b, c \rangle.
    \end{equation*}
    Since $\langle v_i - b, c \rangle \geq 0$ and $\gamma_i > 0$
    for all $i \in [k]$, it follows that $\langle v_i - b, c \rangle = 0$
    for all $i \in [k]$.
    Thus, $v_1, \ldots, v_k$ also belong to $F$.
    Since $F$ is the smallest face containing $a$ and $b$, it
    follows that $a$ and $b$ are not the vertices of an edge of $P$
    since $F$ also contains $v_1$.

    Suppose $a$ and $b$ are not the vertices of an edge of $P$
    and let $F$ denote the smallest face of $P$ containing $a$ and $b$.
    Denote the vertices of $F$ by $a, b, v_1, \ldots, v_k$ with $k \geq 1$
    (if $k = 0$, then the only vertices of $F$ are $a$ and
    $b$, a contradiction to $F$ not being an edge).
    Since $\frac{1}{2}\left(a + b\right)$
    belongs to the relative interior of
    $F$, there exist $\lambda_a, \lambda_b, \lambda_1, \ldots, \lambda_k > 0$
    such that
    $\lambda_a + \lambda_b + \lambda_1 + \cdots + \lambda_k = 1$
    and
    \begin{equation*}
        \frac{1}{2}\left(a + b\right)
        = \lambda_a a + \lambda_b b + \lambda_1 v_1 + \cdots + \lambda_k v_k.
    \end{equation*}
    Since $k \geq 1$, we cannot have $\lambda_a \geq 1/2$ and $\lambda_b \geq
    1/2$.
    If $\lambda_a < 1/2$, then $0 < 1 - 2 \lambda_a$ and so
    \begin{equation*}
            a - b = \frac{2 \lambda_1}{(1 - 2 \lambda_a)} (v_1 - b) + \cdots + \frac{2 \lambda_k}{(1 - 2 \lambda_a)} (v_k - b).
    \end{equation*}
    Set $\gamma_i = \frac{2 \lambda_i}{(1 - 2 \lambda_a)}$.
    If $\lambda_a \ge 1/2$, then $\lambda_b < 1/2$, and so we can swap the roles
    of $a$ and $b$.
\end{proof}

We now apply
the following lemma that applies for any $0/1$-polytope.
\begin{lemma}
    \label{sandwich-lemma}
Let $\mathcal C$ be any set of subsets of $X$.
    If $e_A, e_B, e_{C_1}, \ldots, e_{C_k}$ are distinct vertices
of $\Poly_{\mathcal C}$
    and $e_A - e_B = \sum_{i=1}^k \gamma_i(e_{C_i} - e_B)$
    with $\gamma_1, \dots, \gamma_k > 0$, then
    $A \cap B\subseteq C_i \subseteq A\cup B$
    for all $i \in [k]$.
\end{lemma}

\begin{proof}
    Suppose $x \in A \cap B$. Then
    $
        0
        = \langle e_A - e_B, e_x \rangle
        = \sum_{i=1}^k \gamma_i \left(\langle e_{C_i}, e_x \rangle - 1\right),
    $
    which implies $\langle e_{C_i}, e_x \rangle = 1$ for all $i \in [k]$,
    since $\langle e_{C_i}, e_x \rangle - 1 \leq 0$ and $\gamma_i > 0$.
    Hence, $x \in C_i$ for all $i\in [k]$.

    To prove $C_i \subseteq A \cup B$, suppose $x \notin A \cup B$.
    Since
    $ 
        0
        = \langle e_A - e_B, e_x \rangle
        = \sum_{i=1}^k \gamma_i \langle e_{C_i}, e_x \rangle,
    $ 
    each $\langle e_{C_i}, e_x \rangle \geq 0$,
    and $\gamma_1, \ldots, \gamma_k > 0$,
    it follows that $\langle e_{C_i}, e_x \rangle = 0$ for all $i \in [k]$.
\end{proof}

\begin{proof}[Proof of Theorem~\ref{RP1skeleton}, Part~\eqref{RP1skeleton-edges}]
    If there exist $C,D \in \Stab(G)$ with
    $e_A+e_B=e_C+e_D$, then
    \begin{equation*}
        e_A-e_B=e_C+e_D-2e_B=(e_C-e_B)+(e_D-e_B).
    \end{equation*}
    Thus, if $\{C, D\} \neq \{A, B\}$, then Lemma~\ref{lem:edge} implies
    $\{e_A,e_B\}$ is not an edge of $\SSP(G)$.

    To prove the converse, argue by contradiction.
    Suppose $\{e_A,e_B\}$ is not an edge and
    suppose the following hypothesis holds:
    \begin{enumerate}[label=(H), topsep=1em]
        \item
            \label{hypothesis}
            there do not exist $C, D$ in $\Stab(G)$ such that $C\neq D$,
            $\{A, B\} \neq \{C, D\}$ and $e_A+e_B=e_C+e_D$.
    \end{enumerate}
    By Lemma~\ref{lem:edge}, there exist $e_{C_1}, \ldots, e_{C_k}$
    different from $e_A$ and $e_B$ and $\gamma_1, \dots, \gamma_k > 0$ such
    that
    \begin{equation}
        \label{eq:face1}
        e_A-e_B = \sum_{i=1}^k \gamma_i(e_{C_i} - e_B).
    \end{equation}
    By Lemma~\ref{sandwich-lemma}, we have, for all $i \in [k]$,
    \begin{equation}\label{eq:TheCi}
        A \cap B \subseteq C_i \subseteq A \cup B.
    \end{equation}
    Let $A'=A \setminus B$ and $B' = B \setminus A$.

    \bigskip
        \noindent
        {\bf Claim:}
             \b{We may assume that  $A' \neq \emptyset$ and $B' \neq \emptyset$:}
              If $A'=\emptyset$, then $A\subset B$. From the fact that $A\ne B$ we can pick $x\in B'$ and we have
            \begin{equation*}
                e_{A} + e_{B}
                = e_{A \cup \{x\}} + e_{B \setminus \{x\}}.
            \end{equation*}
        Since subsets of stable sets are stable, we have that
            $A \cup \{x\}, B \setminus \{x\}\subseteq B \in \Stab(G)$.
            Hence,
            $e_{A} + e_{B}  = e_{A \cup \{x\}} + e_{B \setminus \{x\}}$
            contradicts \ref{hypothesis} unless
            $\{A, B\} = \{A \cup \{x\}, B \setminus \{x\}\}$,
            and thus $B = A \cup \{x\}$. Using this information in Equation~\eqref{eq:TheCi} we obtain
            $$ A= A\cap B \subseteq C_i \subseteq A\cup B= B= A\cup\{x\}.$$
            This implies that $C_i=A$ or $C_i=A\cup\{x\}=B$, a contradiction to the choices of $C_i$ in \eqref{eq:face1}.
            Therefore we must have $A'\neq \emptyset$. The argument for $B' \neq \emptyset$ is similar.

       Given that $A' \neq \emptyset$ and $B' \neq \emptyset$, for each $x\in A'$, let
            \begin{equation*}
                B'_x = \{ b' \in B' : \text{$x$ and $b'$ are adjacent in $G$} \}.
            \end{equation*}
            We divide the rest of the proof into a series of steps.

            \begin{enumerate}[label=\textbf{(\alph*)}, wide, labelindent=0pt, itemsep=12pt, topsep=12pt]
                \item
                    \label{5a} First, we prove that \b{$B'_x \neq \emptyset$.}
                    Suppose $B'_x = \emptyset$. Then
                    $B \cup \{x\} \in \Stab(G)$, because:
                    \begin{itemize}[noitemsep, topsep=0pt, label={--}]
                        \item
                            $b$ and $b'$ are not adjacent
                            for distinct $b, b' \in B$, since $B \in \Stab(G)$;

                        \item
                            $x$ and $b$ are not adjacent
                            for $b \in B \setminus B' = A \cap B$,
                            since $x, b \in A$ and $A \in \Stab(G)$;

                        \item
                            $x$ and $b'$ are not adjacent
                            for $b' \in B'$, since $B'_x = \emptyset$.
                    \end{itemize}
                    Also, $e_A, e_B, e_{A \setminus \{x\}}, e_{B \cup \{x\}}$
                    are distinct: otherwise, $B = A \setminus \{x\}$,
                    contradicting the assumption that $B' \neq \emptyset$.
                    But then
                    $e_A + e_B = e_{A \setminus \{x\}} + e_{B \cup \{x\}}$
                    contradicts \ref{hypothesis}.

                \item
                    \label{5b} Next, we prove that \b{$B' = \bigcup_{x\in A'} B'_x.$}
                    Suppose there exists
                    $b' \in B' \setminus \bigcup_{x\in A'} B'_x$.
                    Then $A \cup \{b'\} \in \Stab(G)$, since:
                    \begin{itemize}[noitemsep, topsep=0pt, label={--}]
                        \item
                            $a$ and $a'$ are not adjacent
                            for distinct $a, a' \in A$, since
                            $A \in \Stab(G)$;

                        \item
                            $a$ and $b'$ are not adjacent
                            for $a \in A \setminus A' = A \cap B$,
                            since $a, b' \in B$ and $B \in \Stab(G)$;

                        \item
                            $a'$ and $b'$ are not adjacent
                            for $a' \in A'$,
                            since $b' \notin B'_{a'}$.
                    \end{itemize}
                    Also, $e_A, e_B, e_{A \cup \{b'\}}, e_{B \setminus \{b'\}}$
                    are distinct: otherwise, $A = B \setminus \{b'\}$,
                    contradicting the assumption that $A' \neq \emptyset$.
                    But then
                    $e_A + e_B = e_{A \cup \{b'\}} + e_{B \setminus \{b'\}}$
                    contradicts \ref{hypothesis}.

                \item
                    \label{5c} We prove that for each $x \in A'$ and each $C_i$ in \eqref{eq:face1}, we have
                    $$\b{\big( x \in C_i \text{ and }  B'_x \cap C_i = \emptyset\big)
                      \text{ or } \big(x \notin C_i \text{ and }B'_x \subseteq C_i\big). }$$
               By definition, $b \in B'_x$ if and only if $x$ and $b$ are adjacent in $G$.
                    Hence, $b$ and $x$ cannot both belong to the same stable set.
                    So, if $x \in C_i$, then $b \notin C_i$ for all $b \in B'_x$;
                    that is, $B'_x \cap C_i = \emptyset$.

                    Let $x \in A'$.
                    Then $x \notin B$ and so
                    by Equation~\eqref{eq:face1},
                    \begin{equation}
                        \label{eq:rel1}
                        1
                        = \langle e_A - e_B, e_x \rangle
                        = \sum_{i=1}^k \gamma_i \langle e_{C_i} - e_{B}, e_x \rangle
                        = \sum_{i=1}^k \gamma_i \langle e_{C_i}, e_x \rangle
                        = \sum_{\substack{1 \leq i \leq k \\ x \in C_i}} \gamma_i.
                    \end{equation}
                    For $b \in B'_x$,
                    Equations~\eqref{eq:face1} and \eqref{eq:rel1},
                    together with the fact that $x \in C_i$ implies
                    $b \notin C_i$,
                    \begin{equation}\label{eq:coeffb}
                        \begin{aligned}
                            -1
                            &= \langle e_A - e_B, e_b \rangle
                            =
                            \sum_{\substack{1 \leq i \leq k \\ x \in C_i}}
                                \gamma_i \langle e_{C_i} - e_B, e_b \rangle
                            +
                            \sum_{\substack{1 \leq i \leq k \\ x \notin C_i}}
                                \gamma_i \langle e_{C_i} - e_B, e_b \rangle
                            \\
                            &=
                            \sum_{\substack{1 \leq i \leq k \\ x \in C_i}} -\gamma_i
                            +
                            \sum_{\substack{1 \leq i \leq k \\ x \notin C_i}}
                            \gamma_i \left(\langle e_{C_i}, e_b \rangle - 1 \right)
                            =
                            \quad{-1}
                            + \sum_{\substack{1 \leq i \leq k \\ x \notin C_i}}
                            \gamma_i \left(\langle e_{C_i}, e_b \rangle - 1 \right).
                        \end{aligned}
                    \end{equation}
                    Since each $\gamma_i > 0$, it follows that
                    $b \in C_i$ for all $i$ such that $x \notin C_i$.
                    Hence, $B'_x \subseteq C_i$ for all $1 \leq i \leq k$ such
                    that $x \notin C_i$.

                \item  \label{5d}
                For each $1 \leq i \leq k$, we have
                $$\b{                        C_i =
                        \big(A \cap B\big)
                        \cup
                        \big(A'\cap C_i)
                        \cup
                        \Big( \bigcup_{y\in A' \setminus C_i}B'_{y} \Big).
                 }
               $$
                 Fix $i$ and use \ref{5b} to write
                    \begin{equation}\label{eq:Bdecomposition}
                        B' = \big(\bigcup_{x\in A\cap C_i} B'_{x}\big)
                        \cup \big( \bigcup_{y\in A\setminus C_i} B'_{y}\big)
                    \end{equation}
                 From \ref{5c}, if  $x\in A\cap C_i$, then $B'_{x} \cap C_i = \emptyset$, and if $y\in A\setminus C_i$, then $B'_{y_t} \subseteq C_i$.
                 Hence $B'\cap C_i =  \bigcup_{y\in A' \setminus C_i}B'_{y}$.
                 Using this and Equation~\eqref{eq:TheCi},
                    we obtain the desired results:
                    \begin{equation*}
                     \begin{aligned}
                       C_i&= (A\cup B)\cap C_i = \big((A\cap B)\cup A' \cup B'\big)\cap C_i
                             = (A\cap B) \cup (A'\cap C_i) \cup (B' \cap C_i)\\
                             &= (A\cap B) \cup (A'\cap C_i) \cup \big(\bigcup_{y\in A' \setminus C_i}B'_{y}\big).
                     \end{aligned}
                    \end{equation*}

                \item For $1 \leq i \leq k$, we have \b{$A'\cap C_i\ne \emptyset$} and \b{$A'\setminus C_i\ne \emptyset$}.
                    \label{5e}

           If $A'\setminus C_i= \emptyset$, then $A'\subseteq C_i$ and, using \ref{5d}, we have $C_i = (A \cap B) \cup (A'\cap C_i)= (A \cap B) \cup A' = A$,
                    contradicting that $C_i$ and $A$ are distinct.
                    Similarly, if $A'\cap C_i= \emptyset$, then   from \ref{5d} and \eqref{eq:Bdecomposition} we have $C_i = (A \cap B) \cup \bigcup_{y\in A' \setminus C_i}B'_{y}= (A \cap B) \cup B' = B$,
                    contradicting that $C_i$ and $B$ are distinct.

                \item
                    \label{5f}

                   For each $i$ such that $1 \leq i \leq k$,
                   the following two sets are stable:
                    \begin{equation*}
                      \b{  \begin{aligned}
                            C &= \big(A \cap B\big) \cup (A'\cap C_i) \cup  \big(\bigcup_{y\in A' \setminus C_i}B'_{y}\big)
                            \\
                            D &= \big(A \cap B\big) \cup (A' \setminus C_i) \cup  \big(\bigcup_{x\in A'\cap C_i}B'_{x}\big)\,.
                        \end{aligned}}
                    \end{equation*}
                    From \ref{5d} the set $C=C_i$ is stable by choice of $C_i$. To show that $D$ is stable, let
                    $u$ and $v$ be distinct elements of $D$.

                    \begin{itemize}[noitemsep, topsep=0pt, label={--}]
                        \item
                            Since $(A \cap B) \cup (A' \setminus C_i) \subseteq A$
                            and $(A \cap B) \cup \big(\bigcup_{x\in A'\cap C_i}B'_{x}\big) \subseteq B$
                            are subsets of stable sets, we have
                            that $u$ and $v$ are not adjacent
                            if $u$ and $v$ both belong to any one of these
                            sets.

                        \item
                            Assume that $u \in A' \setminus C_i$ and  $v \in B'_{x}$,
                            for some $x\in A'\cap C_i$. If $u$ and $v$ are adjacent,
                            then $v \in B'_{u}$.
                            But $B'_{u} \cap B_{x} = \emptyset$
                            since, from \ref{5c}, $B'_{u} \subseteq C_i$
                            and $B'_{x} \cap C_i = \emptyset$, a contradiction.
                            Hence, $u$ and $v$ are not adjacent.
                    \end{itemize}

                \item
                    \label{5g} Fix $i$ such that $1 \leq i \leq k$.
                    Using the sets $C$ and $D$ defined in \ref{5f} we have \b{$e_A + e_B=e_C + e_D$}.

                    First we remark that the decomposition \eqref{eq:Bdecomposition} is disjoint since for any $x\in A\cap C_i$ and any $y\in A\setminus C_i$
                    we have by \ref{5c} that $B'_x\cap C_i=\emptyset$ and $B'_y\subseteq C_i$, therefore $B'_x\cap B'_y=\emptyset$.
                    \begin{equation*}
                        \begin{aligned}
                            e_A + e_B
                            & =
                            \left(
                                e_{A \cap B}
                                + \sum_{x\in A\cap C_i} e_{x}
                                + \sum_{y\in A\setminus C_i} e_{y}
                            \right)
                            +
                            \left(
                                e_{A \cap B}
                                + \sum_{u\in\bigcup_{x\in A'\cap C_i} B'_x} e_{u}
                                + \sum_{v\in\bigcup_{y\in A'\setminus C_i} B'_y}^{l} e_{v}
                            \right)
                            \\
                            & =
                            \left(
                                e_{A \cap B}
                                + \sum_{x\in A\cap C_i} e_{x}
                                + \sum_{v\in\bigcup_{y\in A'\setminus C_i} B'_y}^{l} e_{v}
                            \right)
                            +
                            \left(
                                e_{A \cap B}
                                + \sum_{y\in A\setminus C_i} e_{y}
                                + \sum_{u\in\bigcup_{x\in A'\cap C_i} B'_x} e_{u}
                            \right)\\
                            &=
                            e_C + e_D.
                        \end{aligned}
                    \end{equation*}
                    Claim \ref{5g} contradicts our hypothesis \ref{hypothesis} unless
                    $\{A, B\} = \{C, D\}$.
                    But $C_i=C\in \{A, B\}$ is also a contradiction to our hypothesis on $C_i$.
                    \qedhere
            \end{enumerate}
 \end{proof}

\section{Birkhoff polytope of a relation}
\label{sec:birkhoff}

The \demph{Birkhoff polytope} is defined as the convex hull of the $n\times n$
permutation matrices, where we view each permutation matrix as a vector in
$\RR^{n^2}$.
This polytope is a face of a stable set polytope of a graph, as
we now describe.

Let $G$ be a graph with vertex set
$\{(i,j): 1\le i,j\le n\}$ and with edges connecting
$(i,j)$ and $(k,l)$ if and only if
\begin{equation*}
    i=k
    \quad\text{or}\quad
    j=l.
\end{equation*}
In other words, $(i, j)$ and $(k, l)$ are connected if they index entries of an
$n \times n$ matrix that belong to the same row or to the same column.
Hence, the stable sets of $G$ correspond to selecting entries of an $n \times
n$ matrix with at most one entry from each row and each column. Equivalently,
they correspond to \demph{partial permutations} of $[n]$, or to
\demph{non-attacking rook placements} on an $n \times n$ board.

Since the indicator vectors for the maximal stable sets of $G$ are the
permutation matrices, the Birkhoff polytope is the face of $\SSP(G)$ supported
by the hyperplane consisting of the vectors whose coordinates sum to $n$.

This is similar to the relationship seen in
Sections~\ref{sec:matroid-independence-polytope}
and~\ref{sec:matroid-basis-polytope} between the basis polytope and the
independence polytope of a matroid, respectively. This suggests the following
definition that simultaneously generalizes these two constructions.

\begin{definition}
    Let $G$ be a finite simple graph and let
    $r = \max\{|A| : A \in \Stab(G)\}$.
    The \demph{Birkhoff polytope of $G$} is
    \begin{equation*}
        \defcolor{\BP(G)} = \conv\left\{ e_A : A \in \Stab(G) \text{~and~} |A| = r \right\}.
    \end{equation*}
    The \demph{rank} of $\BP(G)$ is defined to be the number $r$.
\end{definition}

Our characterization of the edges of $\SSP(G)$ also characterizes the edges of
$\BP(G)$.

\begin{theorem}
    \label{BP1skeleton}
    Let $\BP(G)$ be the Birkhoff polytope of a finite simple graph $G$
    and let $r$ denote its rank.
    \begin{enumerate}[noitemsep]
        \item\label{BP1skeleton-vertices}
            The vertex set of $\BP(G)$ is $\{e_A : A \subseteq \Bases(G)\}$,
            where $\Bases(G) = \{A \in \Stab(G) : |A| = r\}$.

        \item\label{BP1skeleton-edges}
            Two distinct vertices $e_A$ and $e_B$ form an edge in $\BP(G)$ if and only if
            for all $C, D \in \Bases(G)$, we have
            \begin{equation*}
                \text{$e_A + e_B = e_C + e_D$ implies $\{A, B\} = \{C, D\}$.}
            \end{equation*}
    \end{enumerate}
\end{theorem}

This follows from the fact that $\BP(G)$ is a face of $\SSP(G)$: it is the
intersection of $\SSP(G)$ with the hyperplane consisting of the vectors whose
coordinates sum to $r$.

\begin{remark}
    \label{remark:maximal-sets-polytope-graph}
    It turns out Theorem~\ref{BP1skeleton} does not hold for the polytope
    constructed using all the stable sets of $G$ that are maximal with respect to
    set inclusion. For an example, consider the graph $G$ in Figure~\ref{fig:graph}.
    The following are all the stable sets of $G$ that are maximal with respect
    to inclusion:
    \begin{equation*}
        \begin{array}{r@{\hskip3pt}c@{\hskip3pt}l@{\hskip15pt}r@{\hskip3pt}c@{\hskip3pt}l@{\hskip15pt}r@{\hskip3pt}c@{\hskip3pt}l}
            A &=& \{1,2,3 \} &
            B &=& \{4,5,6\} &
            C &=& \{7,8,9\}
            \\
            D &=& \{1,5,6\} &
            E &=& \{2,4,6\} &
            F &=& \{3,4,5\}
            \\
            G &=& \{1,8,9\} &
            H &=& \{2,7,9\} &
            I &=& \{3,7,8\}
            \\
            J &=& \{2,3,4,7\} &
            K &=& \{1,3,5,8\} &
            L &=& \{1,2,6,9\}
        \end{array}
    \end{equation*}
    Then in the polytope that is the convex hull of the indicator vectors of
    these sets, we have that $e_A$ and $e_B$ are not adjacent: indeed, since
    \begin{equation*}
        e_A - e_B = (e_D - e_B) + (e_E - e_B) + (e_F - e_B),
    \end{equation*}
    it follows from Lemma~\ref{lem:edge} that $e_A$ and $e_B$ are not adjacent.
    However, there are no other maximal stable sets $A'$ and
    $B'$ distinct from $A$ and $B$ such that $e_A+e_B=e_A'+e_B'$.
    \begin{sagecode}
        sage: V = {1,2,3,4,5,6,7,8,9}
        sage: E = {(1,7),(1,4),(2,8),
        ....:      (2,5),(3,9),(3,6),
        ....:      (7,5),(7,6),(8,4),
        ....:      (8,6),(9,4),(9,5)}
        sage: G = Graph((V, E), format="vertices_and_edges")
        sage: Stab = [A for A in Subsets(V) if G.is_independent_set(A)]
        sage: P = Poset((Stab, attrcall("issubset")))
        sage: M = P.maximal_elements()
        sage: sorted(map(sorted, M))
        [[1, 2, 3],
         [1, 2, 6, 9],
         [1, 3, 5, 8],
         [1, 5, 6],
         [1, 8, 9],
         [2, 3, 4, 7],
         [2, 4, 6],
         [2, 7, 9],
         [3, 4, 5],
         [3, 7, 8],
         [4, 5, 6],
         [7, 8, 9]]

        sage: E = (QQ^len(V)).basis()
        sage: e = lambda A : sum([E[a - 1] for a in A], E[0].parent().zero())
        sage: MP = Polyhedron(vertices=[e(C) for C in M])
        sage: EdgeGraph = MP.graph()
        sage: V = {Set([i + 1 for (i, vi) in enumerate(v) if vi != 0]): v
        ....:                                       for v in MP.vertices()}

        sage: A = Set({1, 2, 3})
        sage: B = Set({4, 5, 6})
        sage: EdgeGraph.has_edge((V[A], V[B]))
        False

        sage: {Set([C, D]) for C in M for D in M if e(A) + e(B) == e(C) + e(D)}
        {{{1, 2, 3}, {4, 5, 6}}}
    \end{sagecode}
    \begin{figure}[htpb]
        \centering
        \begin{tikzpicture}[scale=1]
            \node [draw, circle, minimum size=6mm] (v0) at (0cm,  1cm ) {$1$};
            \node [draw, circle, minimum size=6mm] (v2) at (5cm,  1cm ) {$3$};
            \node [draw, circle, minimum size=6mm] (v6) at (1.5cm,-1cm) {$7$};
            \node [draw, circle, minimum size=6mm] (v5) at (3.5cm,-1cm) {$6$};
            \node [draw, circle, minimum size=6mm] (v3) at (1.5cm,3cm ) {$4$};
            \node [draw, circle, minimum size=6mm] (v8) at (3.5cm,3cm ) {$9$};
            \node [draw, circle, minimum size=6mm] (v4) at (3.5cm,1cm ) {$5$};
            \node [draw, circle, minimum size=6mm] (v7) at (1.5cm,1cm ) {$8$};
            \node [draw, circle, minimum size=6mm] (v1) at (2.5cm,2cm ) {$2$};
            \draw [black, line width=3pt] (v0) to (v3);
            \draw [black, line width=3pt] (v0) to (v6);
            \draw [black, line width=3pt] (v1) to (v4);
            \draw [black, line width=3pt] (v1) to (v7);
            \draw [black, line width=3pt] (v2) to (v5);
            \draw [black, line width=3pt] (v2) to (v8);
            \draw [black, line width=3pt] (v3) to (v7);
            \draw [black, line width=3pt] (v3) to (v8);
            \draw [black, line width=3pt] (v4) to (v6);
            \draw [black, line width=3pt] (v4) to (v8);
            \draw [black, line width=3pt] (v5) to (v6);
            \draw [black, line width=3pt] (v5) to (v7);
        \end{tikzpicture}
        \caption{A graph such that the $1$-skeleton of the convex hull of the indicator functions
            of the stable sets that are maximal with respect to inclusion does
            not satisfy Theorem~\ref{BP1skeleton}. For details, see
            Remark~\ref{remark:maximal-sets-polytope-graph}.}
        \label{fig:graph}
    \end{figure}
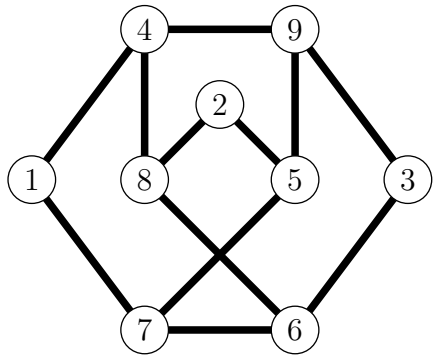
\end{remark}

\section{On $0/1$-polytopes satisfying Theorem~\ref{RP1skeleton}}
\label{hierarchy}

The goal of this section is to study the class of $0/1$-polytopes whose
$1$-skeleton is described by the criterion in Theorem~\ref{RP1skeleton}. These results
are summarized in Figure~\ref{inclusions-polytope-classes}.
More explicitly, a polytope $P$ belongs to this class if and only if $P$ satisfies the following condition:
\begin{equation*}
    \label{1-skeleton-condition}
    \begin{array}{l}
        \text{two distinct vertices $v$ and $u$ form an edge of $P$ if and only if} \\
        \text{there exists a unique way to write $v + u$ as the sum of two vertices of $P$.}
    \end{array}
    \tag{E}
\end{equation*}

We begin with a combinatorial reformulation of the condition~\eqref{1-skeleton-condition}.
\begin{lemma} \label{rem:using-E}
Let $\mathcal C$ be any set of subsets of $X$. If $\Poly_{\mathcal C}$ satisfies condition~\eqref{1-skeleton-condition},
then we can determine the $1$-skeleton combinatorialy as follows. For every pair $\{A,B\}\subseteq {\mathcal C}$ we compute
    $$\chi_{\mathcal C}(\{A,B\})=(A\cap B, A\cup B).$$
Then $\{A,B\}$ is an edge of $P_{\mathcal C}$ if and only if $\chi^{-1}_{\mathcal C}(A\cap B, A\cup B) = \big\{\{A,B\}\big\}$.
\end{lemma}


\subsection{Stable set polytopes and property (E)}
By Theorem~\ref{RP1skeleton}, all stable set polytopes satisfy
\eqref{1-skeleton-condition}, but there are $0/1$-polytopes satisfying
\eqref{1-skeleton-condition} that are not stable set polytopes.
For example, consider the matroid independence polytope
\begin{equation} \label{ex:TrucatedCube}
         \Poly_{\left\{\emptyset,\{1\},\{2\},\{3\},\{1,2\},\{1,3\},\{2,3\}\right\}}\,,
      \end{equation}
which is the cube in $\RR^3$ with the vertex $e_1 + e_2 + e_3$ removed.
In addition, not all $0/1$-polytopes satisfy condition \eqref{1-skeleton-condition}.
 An instance of this is the polytope of
Remark~\ref{remark:maximal-sets-polytope-graph}.
These examples establish the following strict inclusions (see also
Figure~\ref{inclusions-polytope-classes}):
\begin{equation*}
    \text{stable set polytopes}
    \quad \subsetneq \quad
    \text{$0/1$-polytopes satisfying \eqref{1-skeleton-condition}}
    \quad \subsetneq \quad
    \text{$0/1$-polytopes}.
\end{equation*}

\subsection{Partition matroid polytopes (intersection of stable set polytopes and matroid independence polytopes)}
\label{ssec:PartMatroid}
Our next result states that a $0/1$-polytope is both a stable set polytope of
a graph and the independent set polytope of a matroid if and only if the graph is a union
of complete graphs or equivalently, if and only if the matroid is a direct sum of rank 1 uniform
matroids; such matroids are called \demph{partition matroids}.

\begin{proposition}
    \label{stable-set-polytopes-cap-matroid-polytopes}
    Let $G$ be a finite simple graph.
    Then $\SSP(G)$ is the independent set polytope of a matroid if and only if $G$ is
    a union of complete graphs.
\end{proposition}

\begin{proof}
    ($\Leftarrow$) First assume that $G = K_n$.
    Then $\SSP(G) = \conv\{0, e_1, \ldots, e_{n}\}$.
    Thus, $\SSP(G)$ is the independent set polytope of the uniform matroid
    $U_{1,n}$ whose independent sets are the subsets of $[n]$ that contain at
    most $1$ element.
    Next, if $G$ is the disjoint union of two complete graphs $K_a$ and $K_b$,
    then $\SSP(G) = \SSP(K_a) \times \SSP(K_b)$, and hence $\SSP(G)$ is the
    independent set polytope of the matroid $U_{1,a}\oplus U_{1,b}$. The
    general case follows by induction.

    ($\Rightarrow$)
    Suppose that $G$ is a graph with vertex set $[n]$ and that the stable sets
    of $G$ satisfy conditions \ref{matroid-axiom-1} and \ref{matroid-axiom-2}
    of the definition of a matroid (see
    Section~\ref{sec:matroid-independence-polytope}).
    Write $G=G_1\cup\cdots\cup G_r$ as the union of its connected components.
    If the number of vertices of $G_i$ is less than $3$, then it is a complete
    graph ($K_1$ or $K_2$), so consider a connected component $G_i$ with at
    least $3$ vertices. By relabelling, we can assume $i = 1$.

    If $G_1$ is not a complete graph, then there exists three vertices $\{i_1,
    i_2, i_3\}$ of $G$ such that $\{i_1,i_2\}$ and $\{i_2,i_3\}$ are edges
    of $G$ while $\{i_1, i_3\}$ is not.
    It follows that both $A=\{i_2\}$ and $B=\{i_1,i_3\}$ are stable sets of
    $G$.
    By \ref{matroid-axiom-2}, there is an element $b \in B\setminus A$ such
    that $A\cup \{b\}$ is also a stable set, but this is not the case.
    This contradiction implies $G_1$ is a complete graph.
\end{proof}

Most  graphs are not the union of complete graphs, which establishes the containment
\begin{equation*}
    \text{partition matroids polytopes}
    \quad \subsetneq \quad
     \text{stable set polytopes}.
\end{equation*}

\subsection{Matroid polytopes and property (E)}
Section~\ref{ssec:PartMatroid} shows that for most matroids, its independence polytope is not a stable set polytope.
Yet property \eqref{1-skeleton-condition} holds for the $1$-skeleton of matroid basis polytopes and the matroid
independence polytopes as we now prove.

\begin{theorem}
    \label{matroid-polytopes-satisfy-(E)}
    Let $P_M$ be the independence polytope of a matroid $M$.
    Two distinct vertices $e_A$ and $e_B$ of $P_M$ form an edge of $P_M$ if and
    only if there exists a unique way to write $e_A + e_B$ as the sum of two
    vertices of $P_M$.
\end{theorem}

\begin{proof}
    ($\Rightarrow$)
    We prove the contrapositive.
    If $e_A + e_B = e_C + e_D$ with $\{A, B\} \neq \{C, D\}$,
    then
    \begin{equation*}
        e_A - e_B = (e_A + e_B) - 2 e_B
        = (e_C - e_B) + (e_D - e_B),
    \end{equation*}
    which, by Lemma~\ref{lem:edge}, implies that $e_A$ and $e_B$ are not the
    vertices of an edge.

    ($\Leftarrow$)
    We provide a proof by contradiction.
    Let $e_A$ and $e_B$ be two vertices of $P_M$ such that $e_A + e_B$
    can be written as a sum of two vertices of $P_M$ in a unique way;
    and suppose that $e_A$ and $e_B$ are \emph{not} the vertices of an edge of
    $P_M$.
    Combining Lemma~\ref{lem:edge} and Lemma~\ref{sandwich-lemma},
    there exist vertices $e_{C_1}, \ldots, e_{C_k}$
    of $P_M$, each distinct from $e_A$ and $e_B$,
    and $\gamma_1, \ldots, \gamma_k > 0$ such that
    \begin{equation}\label{adjacent-condition}
        e_A-e_B=\sum_{i=1}^k \gamma_i(e_{C_i}-e_B)
        \qquad\text{and}\qquad
        A \cap B \subseteq C_i \subseteq A \cup B
        \text{~for all $1 \leq i \leq k$}.
    \end{equation}

    \begin{enumerate}[wide, itemsep=1ex, topsep=1ex, label={\textbf{Case \arabic*:}}]
        \item
            $|A| \neq |B|$.
            Without loss of generality, suppose $|A| > |B|$.
            By the matroid axiom \ref{matroid-axiom-2}, there exists $a\in
            A\setminus B$ such that $B\cup \{a\}$ is independent.
            Also, $A\setminus \{a\}$ is independent
            by \ref{matroid-axiom-1}.
            Hence, both $e_{B \cup \{a\}}$ and $e_{A \setminus \{a\}}$
            are vertices of $P_M$ that sum to
            \begin{equation*}
                e_{B \cup \{a\}} + e_{A \setminus \{a\}} =
                (e_{B} + e_{a}) + (e_{A} - e_{a}) =
                e_A + e_B.
            \end{equation*}
            Since there is a unique way to write $e_A + e_B$
            as the sum of two vertices of $P_M$, it follows that
            $A = B \cup \{a\}$ and $B = A \setminus \{a\}$.
            Therefore, all the $C_i$ appearing in
            Equation~\eqref{adjacent-condition} satisfy
            $B \subseteq C_i \subseteq B \cup \{a\}$.
            Thus, $C_i = B$ or $C_i = A$, both of
            which contradict $C_i \neq A, B$.

        \item
            $|A| = |B| = r$.
            By the \emph{Strong Exchange Theorem} (\cite[section 1.5.1]{bgw})
            for any
            $a \in A\setminus B$, there exists
            $b \in B\setminus A$ such that
            $(A \setminus \{a\}) \cup \{b\}$
            and
            $(B \setminus \{b\}) \cup \{a\}$
            are independent sets.
            Hence,
            $e_{(A\setminus \{a\})\cup \{b\}}$
            and
            $e_{(B\setminus \{b\})\cup \{a\}}$
            are vertices of $P_M$
            that sum to $e_A + e_B$.
            Since there is a unique way to write $e_A + e_B$
            as the sum of two vertices, it follows that
            \begin{equation*}
                A = (B\setminus \{b\})\cup \{a\}
                \qquad\text{and}\qquad
                B = (A\setminus \{a\})\cup \{b\}.
            \end{equation*}

            Consider the sets $C_i$ appearing in Equation~\eqref{adjacent-condition}.
            Since $A \cap B \subseteq C_i \subseteq A \cup B = (A \cap B) \cup \{a,b\}$,
            there are two possibilities: either $C_i = A \cap B$ or $C_i
            = (A \cap B) \cup \{a, b\}$ (recall that $C_i \neq A, B$).

            Suppose there exists an $i$ such that $C_i = (A \cap B) \cup \{a, b\}$.
            Since $C_i = A \cup \{b\}$, we have that $e_{C_i}$ and $e_{B \setminus
            \{b\}}$ are vertices of $P_M$ that sum to $e_A + e_B$.
            This implies
            $A = A \cup \{b\}$ (which contradicts $b \in A \setminus B$)
            or
            $A = B \setminus \{b\}$ (which contradicts $|A| = |B|$).
            Thus, no such $i$ exists.

            Therefore, each $C_i$ appearing in Equation~\eqref{adjacent-condition} is
            equal to $A \cap B$, and so
            \begin{equation*}
                e_A - e_B = \gamma (e_{A \cap B} - e_{B})
            \end{equation*}
            for some $\gamma > 0$.
            Substituting $A = (B \setminus \{b\}) \cup \{a\}$
            on the left,
            and $B = (A \cap B) \cup \{b\}$ on the right,
            we obtain $e_{a} - e_{b} = \gamma (- e_{b})$,
            which is absurd since
            $e_a$ and $e_b$ are linearly independent.
            \qedhere
    \end{enumerate}
\end{proof}

One can also find several polytopes that satisfy~\eqref{1-skeleton-condition} but do not come from a matroid nor the stable sets of a graph. An example of this is the polytope
\begin{equation} \label{ex:nonsimplicial}
   \widehat{C} = \Poly_{\left\{\emptyset,\{2\},\{3\},\{4\},\{2,3\},\{3,4\},\{1,2,3\}\right\}}\,.
\end{equation}
These results establish the following strict inclusions from
Figure~\ref{inclusions-polytope-classes}:
\begin{equation*}
    \text{partition matroids polytopes}
    \quad \subsetneq \quad
    \text{matroids polytopes}
    \quad \subsetneq \quad
    \text{$0/1$-polytopes satisfying \eqref{1-skeleton-condition}}
\end{equation*}

We end this section by remarking that one can derive from
Theorem~\ref{matroid-polytopes-satisfy-(E)} the description of the $1$-skeleton of
the matroid basis polytope first given in \cite[Thm. 4.1]{ggms1987} and that of
the matroid independence polytope first given by
\cite[Thm. 5.1]{Topkis1984}.

\subsection{Simplicial complex polytopes and property (E)}
\label{ssec:ConjectureSCP}

To complete the justification of the inclusions depicted in
Figure~\ref{inclusions-polytope-classes}, we explore the relationship between
simplicial complex polytopes, matroid polytopes, and $0/1$-polytopes satisfying (E).

Since the collection of independent sets of a matroid
and the collection of stable sets of a graph
are both simplicial complexes,
the class of matroid polytopes and the class of stable set polytopes
are included in the class of simplicial complex polytopes.
The intersection of these two classes is the class
of partition matroid polytopes defined from Section~\ref{ssec:PartMatroid}.

The inclusion of stable set polytopes and matroid polytopes in the class
of simplicial complex polytopes is strict because
$\Poly_{\left\{\emptyset,\{1\}, \{2\},\{3\},\{4\} ,\{1,2\},\{1,3\},\{2,3\}\right\}}$
is a simplicial complex polytope
that is neither a matroid polytope nor a stable set polytope.

Finally, we show that the class of simplicial complex polytopes overlaps with
the class of $0/1$-polytopes satisfying (E), but neither is included in the
other.
The polytope in~\eqref{ex:nonsimplicial} is an example of a $0/1$-polytope
satisfying (E) that is not a simplicial complex polytope (nor the facet of
a simplicial complex polytope).
And the polytope
\begin{equation*}
    \Poly_{\left\{A: A\subseteq \{1,2,6\}, \{3,4,5\}, \{3,4,6\}, \{2,3,5\} \text{ or } \{1,4,5\} \right\}}
\end{equation*}
is a simplicial complex polytope that does not satisfy (E). Indeed, let $A=
\{1,2,6\}$ and $B= \{3,4,5\}$, using Lemma~\ref{rem:using-E} we compute
$\chi^{-1}(\emptyset,\{1,2,3,4,5,6\}) = \big\{\{A,B\}\big\}$. If the
condition~\eqref{1-skeleton-condition} is satisfied it should be an edge of the
polytope, but it is not. This is the smallest possible counter example: all
simplicial complexes for $n<6$ satisfy the
condition~\eqref{1-skeleton-condition}.

\section{On the diameter}
\label{section-on-the-diameter}

The \emph{Hirsch conjecture} asserts that the diameter of every $d$-dimensional
convex polytope $P$ with $n$ facets is at most $n - d$, where the
\demph{diameter} of $P$ is the smallest number $\delta(P)$ such that every pair
of vertices of $P$ are connected in its $1$-skeleton by a shortest path of
length at most $\delta(P)$.
The conjecture remained open for more than fifty years before a counter-example
was found~\cite{Santos2012}. Although it is false in general, it is true for
$0/1$-polytopes \cite{Naddef1989}.
Here we provide a improved bounds for the diameter for the polytopes $\BP(G)$
and $\SSP(G)$.

\subsection{A bound on the diameter of $\BP(G)$}
Our first step is to prove a technical result that is inspired by
the basis exchange property for matroids.

\begin{lemma}\label{quasimatroid}
    If $\mathcal{I}$ is a family of equisized finite sets
    and $A, B \in \mathcal{I}$, then for every $i \in A \setminus B$,
    there exist $E \subseteq A\setminus B$, $F \subseteq B\setminus A$
    satisfying:
    \begin{enumerate*}
        \item $|E|=|F|$;
        \item $i\in E$;
        \item $(A \setminus E) \cup F \in \mathcal{I}$; and
        \item if $e_A + e_{(A\setminus E)\cup F}=e_M + e_N$ with
            $M,N\in\mathcal{I}$, then $\{A, (A\setminus E) \cup F\}=\{M,N\}$.
    \end{enumerate*}
\end{lemma}

\begin{proof}
    We proceed by induction on $m = |A \setminus B|$.
    Suppose first that $m = 1$.
    Then $A \setminus B=\{i\}$ and $B \setminus A = \{j\}$ for some $i, j \in [n]$.
    The sets $E = \{i\}$ and $F = \{j\}$ satisfy the conditions:
    (1) and (2) are immediate; (3) holds because
    $(A \setminus E) \cup F = (A \setminus \{i\}) \cup \{j\} = B$, which is in $\mathcal{I}$;
    and (4) holds because if
    $e_{M} + e_{N} = e_{A} + e_{B} = 2 e_{A \cap B} + e_{i} + e_{j}$,
    then $M \cap N = A \cap B$,
    from which it follows that $M$ is $(A \cap B) \cup \{i\} = A$
    or $(A \cap B) \cup \{j\} = B$.

    For the induction hypothesis, we suppose the result holds for all choices
    of $A, B \in \mathcal{I}$ with $|A \setminus B| < m$;
    we will prove the result also holds for all choices of $A, B \in \mathcal{I}$
    with $|A \setminus B| = m$.

    Let
    $A, B \in \mathcal{I}$ with $|A \setminus B| = m$, and fix $i \in A \setminus B$.
    Then $E = A \setminus B$ and $F = B \setminus A$
    satisfy conditions (1)--(3), because
    \begin{enumerate}
        \item
            $|E| = |A \setminus B| = |A| - |A \cap B| = |B| - |A \cap B| = |F|$,
            because $|A| = |B|$;
        \item
            $i \in E$, because $i$ is an element of $A \setminus B = E$;
        \item
            $(A \setminus E) \cup F = (A \setminus (A \setminus B)) \cup (B \setminus A) = B$,
            which belongs to $\mathcal{I}$.
    \end{enumerate}
    If condition (4) is also satisfied, then there is nothing more to do.
    If condition (4) does not hold, then we can replace $E$ and $F$ by two
    other sets that satisfy all the conditions, as follows.

    Suppose condition (4) fails. Then there exist
    $M,N\in\mathcal{I}$ such that $\{M, N\}\neq\{A, B\}$ and
    \begin{equation}
        \label{eM+eN=eA+eB}
        e_{M} + e_{N}
        = e_{A} + e_{B} = 2 e_{A \cap B} + e_{A \setminus B} + e_{B \setminus A}.
    \end{equation}
    Since $i \in A \setminus B$, it follows from \eqref{eM+eN=eA+eB}
    that $i$ belongs to $M$ or $N$, but not both.
    Without loss of generality, we assume $i \in M$ and $i \notin N$.
    It also follows from \eqref{eM+eN=eA+eB} that
    every element of $A \cap B$ is in $M$ and $N$,
    and so
    \begin{equation}
        \label{inclusions-of-intersections}
        \text{$A \cap B \subseteq A \cap M$ and $A \cap B \subseteq A \cap N$.}
    \end{equation}
    Both of these inclusions are strict. Indeed, the first is strict because
    $i \in A \cap M$ and $i \notin B$.
    To see why the second is strict, note that it suffices to show that
    $N$ contains an element of $A \setminus B$.
    Suppose the contrary.
    Then $N \subseteq (A \cap B) \cup (B \setminus A) = B$,
    which implies that $N = B$ because $|N| = |B|$.
    This in turn implies that $M = A$, which contradicts $\{M, N\} \neq \{A, B\}$.

    Since $A \cap B \subsetneq A \cap N$, it follows that
    $|A \setminus N| = |A| - |A \cap N| < |A| - |A \cap B| = |A \setminus B| = m$.
    By the induction hypothesis applied to $A, N \in \mathcal{I}$ and $i \in A \setminus N$,
    there exist $E, F$ satisfying:
    \begin{enumerate}[start=0]
        \item
            $E \subseteq A \setminus N$ and $F \subseteq N \setminus A$;
        \item
            $|E| = |F|$;
        \item
            $i \in E$;
        \item
            $(A \setminus E) \cup F \in \mathcal{I}$; and
        \item
            if $e_{A} + e_{(A\setminus E) \cup F} = e_{M'} + e_{N'}$
            with $M', N' \in \mathcal{I}$,
            then $\{M', N'\} = \{A, (A \setminus E) \cup F\}$.
    \end{enumerate}
    It remains to show that $E \subseteq A \setminus B$ and $F \subseteq B \setminus A$.
    These follow from \eqref{inclusions-of-intersections}
    and \eqref{eM+eN=eA+eB}:
    \begin{gather*}
        E \subseteq A \setminus N
        = A \setminus (A \cap N) \subseteq A \setminus (A \cap B) = A \setminus B
        \\
        F \subseteq N \setminus A \subseteq (A \cup B) \setminus A
        = B \setminus A.
        \qedhere
    \end{gather*}
\end{proof}

The first application of the lemma is a bound on the diameter
of the Birkhoff polytope of a graph.

\begin{theorem}\label{theo:diameterB}
    Let $G$ be a finite simple graph and let $\BP(G)$ be the corresponding Birkhoff polytope. Let $r= \max\{|A| : A \in \Stab(G)\}$.
    Then $$\delta(\BP(G))\leq r.$$
    That is, the diameter of the Birkhoff polytope of $G$ is at most its rank.
\end{theorem}

\begin{proof} Recall that $\BP(G)=\conv\{e_A : A \subseteq \Bases(G)\}$,
            where $\Bases(G) = \{A \in \Stab(G) : |A| = r\}$.
    Let $A,B\in \Bases(G)$ and fix $i\in A\setminus B$. By Lemma
    \ref{quasimatroid}, we can find $E\subseteq A\setminus B$ and $F\subseteq
    B\setminus A$ such that $i\in E$, $A_1=(A\setminus E)\cup F\in\Bases(G)$
    and for all $M,N\in\Bases(G)$ with $e_A+ e_{(A\setminus E)\cup
    F}=e_M +e_N$, we have $\{A, (A\setminus E) \cup F\}=\{M,N\}$.

    By Theorem \ref{BP1skeleton}, this condition is to say that $\{e_A,e_{A_1}\}$ is an
    edge in $\BP(G)$. Since $i\notin A_1$, we have $A\cap B\subset A_1\cap B$
    and the inclusion is strict. We can then repeat this process with $\{e_{A_1},e_B\}$
    to find $A_2\in\Bases(G)$ such that $\{e_{A_1},e_{A_2}\}$ is an edge in $\BP(G)$ and
    $A_1\cap B\subset A_2\cap B$ with strict inclusion.

    If we continue this process, we get $A\cap B\subset A_1\cap B\subset\dots
    \subset A_\ell\cap B=B$. Since all inclusions are strict, this process must
    terminate in at most $|B\setminus A|$ steps, which is at most
    $r=\max\{|A|:A\in\text{Stab}(G)\}$. Therefore, the distance from $e_A$ to $e_B$
    is at most $r$, via the edges $\{e_{A},e_{A_1}\}, \{e_{A_1},e_{A_2}\},\dots,\{e_{A_2},e_{A_\ell}\}$.
\end{proof}

\subsection{A bound on the diameter of $\SSP(G)$}

To bound the diameter of $\SSP(G)$, we need the following technical result.

\begin{lemma}\label{BaseChange}
    Let $A$ and $B$ be two stable sets of $G$, written as
    $A=\{a_1,\dots,a_k,c_1,\dots,c_\ell\}$ and
    $B=\{b_1,\dots,b_m,c_1,\dots,c_{\ell}\}$ where $a_i\neq b_j$ for all $i$
    and $j$.
    Then there exists a third stable set $C$ of $G$ such that $\{e_A, e_C\}$ is
    an edge in $\SSP(G)$,
    $\{c_1,\dots,c_\ell\}\subseteq C\subseteq A \cup B$
    and $C\cap\{b_1,\dots,b_m\}\neq\emptyset$.
\end{lemma}

\begin{proof}
    If $\{e_A,e_B\}$ is an edge in $\SSP(G)$, then we set $C=B$ and we are done.

    Otherwise, by Theorem \ref{RP1skeleton}, there exists a pair of vertices
    $C_1,D_1$ in $\SSP(G)$ such that $e_A+ e_B=e_{C_1} +e_{D_1}$ and
    $\{A,B\}\neq\{C_1,D_1\}$. Clearly $\{c_1,\dots,c_\ell\}\subset C_1\cap
    D_1$.

    If $A\subset C_1$, then there must exist some $b_i\in C_1$. We can set
    $C=A\cup \{b_1\}\in \SSP(G)$ and we are done.

    Therefore, without loss of generality, we can assume that $C_1\cap
    \{b_1,\dots,b_m\}\neq \emptyset$ and $C_1\cap
    \{a_1,\dots,a_k\}\neq\emptyset$. If $(A,C_1)$ is not an edge, we continue
    this process and get $C_2,D_2$ and so on. In each step, we have the
    following conditions
    \begin{enumerate}
        \item $A\cap C_1\subsetneq A\cap C_2\subsetneq\cdots\subsetneq A\cap C_t$, and
        \item $C_i\cap\{b_1,\dots,b_m\}\neq\emptyset$.
    \end{enumerate}
    Therefore, this process will eventually terminate at some $C_t$, and we
    find an edge that is either $\{e_{A},e_{C_t}\}$ or $\{e_A,e_{A\cup\{b_i\}}\}$ for some
    $b_i\in C_t$.
\end{proof}

Finally, we prove an upper bound for the diameter of $\SSP(G)$ in analogy with Theorem~\ref{theo:diameterB}.
\begin{proposition} \label{prop:diameter}
    If the largest size of a stable set in $G$ is $r$, then the diameter of
    $\SSP(G)$ is at most $r$.
    \end{proposition}

\begin{proof}
    Given two vertices $e_A,e_B$ in $\SSP(G)$, Let $A=\{a_1,\dots,a_m\}$,
    $B=\{b_1,\dots,b_\ell\}$. If $|A|+|B|\leq r$, then we can find a path
    $e_A,e_{A\setminus \{a_1\}},\dots,e_{\{a_m\}},0,e_{\{b_1\}},\dots,e_B$ of
    length $|A|+|B|$ that connects $e_A$ and $e_B$.

    Otherwise, by Lemma \ref{BaseChange}, we can find a path
    $e_A,e_{A_1},\dots,e_{A_t}$ such that $A\cap B\subset A_1\cap
    B\subset\cdots\subset A_t\cap B$ and $B\subseteq A_t$. And we have another
    path $e_{A_t},\dots, e_{A_{t+s-1}},e_{B}$ by removing the elements in $A_t\setminus B$.
    Since we have $t\leq \ell$, $|A_t|\leq r$ and $s\leq r-\ell$, the distance
    from $e_A$ to $e_B$ is at most $r$.
\end{proof}

\begin{remark}
    A result similar to Proposition~\ref{prop:diameter} also holds for the
    independence polytope $P_M$ of a matroid $M$: explicitly, we have
    $\delta\big(P_M\big)\le r$, where $r$ is the rank of $M$
    (that is, the largest size of an independent set).
    This follows by mimicking the proof of Proposition~\ref{prop:diameter}
    and replacing every use of Lemma~\ref{quasimatroid} by the basis exchange
    property of $M$. This result is well-known, so we do not include all the
    details.
\end{remark}

\subsection{Relationship with the Hirsch conjecture}

    We end this section by describing the relationship between the bounds
    proved in the last two subsections and the bound from the statement of the
    Hirsch conjecture.

    Let $G=(V,E)$ be a simple graph. In this context, the Hirsch conjecture
    asserts an upper bound on the diameter of the associated stable set
    polytope:
    \begin{equation*}
        \delta\big(\SSP(G)\big) \le n - d,
    \end{equation*}
    where $n$ is the number of facets of $\SSP(G)$ and $d = \dim(\SSP(G)) = |V|$.

    For any simple graph $G$, Equation~\eqref{always-facets}
 in Section~\ref{sec:OpenProb}
    describes two
    families of facet-defining inequalities of $\SSP(G)$. Since these
    inequalities are indexed by the vertices and the cliques of $G$, we have
    \begin{equation*}
        d + c \leq n,
    \end{equation*}
    where $d = |V|$ and $c=\big|\Cliq(G)\big|$.
    Moreover, $n = d + c$ if and only if $G$ is a perfect graph.

    On the other hand, since any stable set intersects a clique of $G$ in at
    most one vertex, we have $r \le c$, where $r$ is the largest size of
    a stable set in $G$.
    Then by Proposition~\ref{prop:diameter} we have
    \begin{equation*}
        \delta\big(\SSP(G)\big) \le r \le c \le n-d.
    \end{equation*}
    It turns out that $r < c$ in general, even when $G$ is a perfect graph.
    Hence, Proposition~\ref{prop:diameter} is an improvement on the Hirsch
    upper bound of $n - d$.

\section{Open problems}\label{sec:OpenProb}

We list here some interesting open problems related to $0/1$-polytopes
that satisfy \eqref{1-skeleton-condition}.

\subsection{Simplicial complex polytopes that satisfy (E)}

In Theorems~\ref{RP1skeleton} and \ref{matroid-polytopes-satisfy-(E)}
we proved that stable set polytopes and matroid polytopes satisfy
property~\eqref{1-skeleton-condition}.
Given that these polytopes belong to the family of simplicial
complex polytopes, it would be much more interesting to have a uniform proof of
these results. However, as we have seen in Section~\ref{ssec:ConjectureSCP},
not all simplicial complexes satisfy~\eqref{1-skeleton-condition},
so the first step in this direction would be the following.
\begin{problem} \label{Prob1}
Find a characterization of the simplicial complex polytopes that satisfy~\eqref{1-skeleton-condition}.
\end{problem}

We point out that one cannot simply adapt the proof of Theorem~\ref{RP1skeleton} as steps
\ref{5a} and \ref{5b} do not hold for all simplicial complex polytopes (or even
matroid polytopes).
To see this, take the polytope $\widehat C$ in \eqref{ex:TrucatedCube}; using
$A=\{1,2\}$ and $B=\{1,3\}$, we obtain $B'_2=\emptyset$ and $B'\ne B'_2$.

Another difficulty arises from a particularity of stable set polytopes:
if there are enough small stable sets, then they can be combined to build
larger stable sets. This behaviour is quite different than what happens for
simplicial complex polytopes and matroid polytopes.

One possible approach is to adapt the proof of
Theorem~\ref{matroid-polytopes-satisfy-(E)} using a stronger version of
Lemma~\ref{quasimatroid}.
To this end, it would be useful to have a characterization of the elements that
can be removed from a simplicial complex satisfying property~\eqref{1-skeleton-condition}
so that the resulting polytope still preserves
property~\eqref{1-skeleton-condition}.
This could afford an inductive approach to Theorem~\ref{RP1skeleton}.

\subsection{The Mihai--Vazerani conjecture}

The \emph{Mihai--Vazerani conjecture for $0/1$-polytopes} asserts that for
every partition $\mathcal{S} \uplus \mathcal{T}$ of the set of vertices
of the polytope, the number of edges between $\mathcal{S}$ and $\mathcal{T}$
is at least $\min(|\mathcal{S}|, |\mathcal{T}|)$.
In the terminology of expander graphs, the conjecture asserts that
the $1$-skeleton of a $0/1$-polytope is a \emph{1-expander graph}.
Although the conjecture is open in general, it holds for stable set
polytopes~\cite{Kaibel01} and matroid polytopes~\cite{ALGV19}.
Since these $0/1$-polytopes satisfy \eqref{1-skeleton-condition},
it is natural to study this conjecture in this context.

\begin{problem} \label{Prob2}
    Suppose $P$ is a $0/1$-polytope satisfying property \eqref{1-skeleton-condition}.
    Determine whether the $1$-skeleton of $P$ is a $1$-expander graph.
    (This holds for stable set and matroid polytopes \cite{Kaibel01, ALGV19}.)
\end{problem}

\subsection{Describing the facets of $0/1$-polytopes}

\begin{problem} \label{Prob3}
Describe the facets of some families of $0/1$-polytopes satisfying~\eqref{1-skeleton-condition}.
A description is known for matroid polytopes~\cite{Edmonds1970} and for stable set polytopes of perfect graphs~\cite{Chvatal1975}.
\end{problem}

There is no known complete description of the facets of the stable set polytope
of an arbitrary graph. In fact, it is most likely an intractable
problem since the problem of finding the size of a maximal stable set of $G$ is
known to be NP-hard. However, some information is known, and we present below
partial descriptions for some polytopes from Section~\ref{sec:polytopes}.

\subsubsection{Some inequalities valid for all stable set polytopes}
\label{section-always-facets}

Padberg~\cite{Padberg1973} proved the following two families of
inequalities define facets of $\SSP(G)$ for any finite graph $G = (V, E)$:
\begin{equation}
    \label{always-facets}
    0 \leq x_v \quad (v \in V)
    \qquad\text{and}\qquad
    \sum_{v \in C} x_v \leq 1 \quad (C \in \Cliq(G)),
\end{equation}
where $\Cliq(G)$ is the set of cliques of a graph $G$.
Chv{\'a}tal proved that these two families constitute a complete description of
the facets if and only if $G$ is a perfect graph \cite[Theorem~3.1]{Chvatal1975}.
(Recall that a graph is \emph{perfect} if for each subgraph $G'$, the
chromatic number of $G'$ is equal to the maximal cardinality of clique of
$G'$.)

\subsubsection{Chain Polytopes and the Nonnesting Partition Polytopes}
\label{facets-NNP}

If $G_P$ is the comparability graph of a partial order $P$, then $\SSP(G_P)$ is
the poset chain polytope introduced by
Stanley~\cite{Stanley1986} (see Section~\ref{sec:chain-polytope}). Stanley
described the facets by noting that the graph $G_P$ is perfect, and so the
facets are given by \eqref{always-facets}: there is one facet for each element
$x$ of the poset; and one facet for each \emph{maximal} chain $C$ of the poset.
In particular, this gives a complete description of all the facets of the
nonnesting partition polytopes $\NNP_n$ defined in Section~\ref{sec:NNP}.

\subsubsection{Bell polytopes of type $A$.}
\label{facets-Bell-polytopes-A}

In J. Pulido's B.~Sc. Thesis \cite{pulidothesis},
it is shown that all the facets of the Bell polytopes defined in
Section~\ref{sec:bell-polytope} are of the form given by~\eqref{always-facets}.
(Note that these polytopes are not chain polytopes of some poset.)
Explicitly, the second family of inequalities are
\begin{equation*}
    \sum_{i < j \leq n} x_{(i,j)} \leq 1 \quad (1 \leq i < n)
    \qquad\text{and}\qquad
    \sum_{1 \leq i < j} x_{(i, j)} \leq 1 \quad (1 < j \leq n).
\end{equation*}

\subsubsection{Bell polytopes of type $B$.}
\label{facets-Bell-polytopes-B}
The Bell polytope of type $B$ was independently studied by Allen~\cite{Allen2017}.
Again, all the facets of the Bell polytopes of type $B$
are described by~\eqref{always-facets}.
Explicitly, the second family of inequalities are
\begin{equation*}
    \sum_{i \leq j \leq n} x_{(i, j)} \leq 1 \quad (1 \leq i \leq n)
    \qquad\text{and}\qquad
    \sum_{1 \leq i \leq j} x_{(i, j)} \leq 1 \quad (1 \leq j \leq n).
\end{equation*}

\subsubsection{Noncrossing partition polytopes}
\label{facets-NCP}
The inequalities in Equation~\eqref{always-facets} are not sufficient to
describe all the facets of the noncrossing partition polytopes $\NCP_n$ (see
Section~\ref{sec:NCP}).
For example, when $n = 6$,
the two families in \eqref{always-facets} account for $15$ facets and $16$
facets, respectively, whereas $\NCP_6$ has $32$ facets.
The missing facet is defined by the hyperplane
\begin{equation*}
    x_{(1, 3)} + x_{(1, 5)} + x_{(1, 6)} + x_{(2, 3)} + x_{(2, 4)} + x_{(2, 5)} + x_{(2, 6)} + x_{(4, 5)} + x_{(4, 6)} + x_{(5, 6)} = 2.
\end{equation*}
Our computations suggest that the facets of $\NCP_m$ are supported by
hyperplanes of the form $\sum_{a \in X} c_a x_a = m$ with $m, c_a \in \NN$.
When $n = 8$, some coefficients $c_a$ are greater than $1$.

\subsubsection{Matroid polytopes} \label{ssec:facets-matroid}
The facets of the independence polytope of a loopless matroid $M$
were first described by Edmonds~\cite{Edmonds1970}.
They admit the following description \cite[Theorem 40.5]{Schrijver03}:
\begin{equation*}
    0 \leq x_v \quad (v \in V)
    \qquad\text{and}\qquad
    \sum_{v \in F} x_v \leq \rank(F) \quad (F \text{ non-empty inseparable flat of } M).
\end{equation*}

\subsubsection{Simplicial complex polytopes}
It would be quite interesting to generalize the description of
the facets in Section~\ref{ssec:facets-matroid} to some simplicial
complex polytopes. More particularly, a stronger version of
Lemma~\ref{quasimatroid} might help describe the facets for pure simplicial
complexes.
For the moment, this seems inaccessible, but we hope to realize progress
fairly soon.

\bibliographystyle{amsplain}
\bibliography{references}

\end{document}